\documentclass[11pt]{article}
\usepackage[margin=1in]{geometry}
\usepackage{amsmath,amssymb,amsthm,booktabs,array,graphicx,microtype}
\usepackage[T1]{fontenc}
\usepackage{lmodern}
\usepackage[colorlinks=true,allcolors=blue]{hyperref}
\usepackage[round,authoryear]{natbib}
\newtheorem{theorem}{Theorem}
\newtheorem{proposition}[theorem]{Proposition}
\newtheorem{corollary}[theorem]{Corollary}

\newcommand{\E}{\mathbb E}
\newcommand{\Pp}{\mathbb P}
\newcommand{\ind}{\mathbf 1}
\newcommand{\KL}{\operatorname{KL}}
\newcommand{\Q}{\mathcal Q}
\newcommand{\F}{\mathcal F}
\newcommand{\G}{\mathcal G}
\newcommand{\pos}[1]{\left(#1\right)_+}

\begin{document}

\title{Multiscale Sequential Auditing of Quantile Regression with Point Masses}
\author{Asad Ullah\\[2pt]
\small University of North Carolina at Charlotte\\
\small\texttt{aullah@uncc.edu}}
\date{}
\maketitle

\begin{abstract}
For discrete outcomes, a valid conditional quantile can have a hit probability different from its nominal level. Errors in different covariate groups can also cancel when residual signs are pooled. We study sequential audits that address both problems by testing the two inequalities that define quantile-set membership. The audit records whether each response falls below, equals, or exceeds its forecast. We derive the information projection for this three-outcome experiment and use continuous betting mixtures within the cells of a covariate tree. Context-tree weighting combines the resulting audits over all prunings, with an explicit finite-path oracle bound. Under iid alternatives with uniformly positive sign probabilities, a countable mixture of refining trees attains the full conditional sign-information rate. A matching upper bound shows that no valid evidence sequence based on the same observations can attain a larger almost-sure exponential rate. We also bound the information lost through covariate coarsening, derive a recurrence for the local mixtures, and discuss quantile selection and confidence-sequence inversion. Simulations show how partition resolution affects detection, and a held-out count-data example illustrates the difference between quantile validity and exact-hit calibration.
\end{abstract}

\noindent\textbf{MSC2020:}
Primary 62L10, 62J02; secondary 62G10.\\
\textbf{Keywords:}
Conditional calibration; context-tree weighting; discrete outcomes;
e-process; quantile regression; sequential inference.

\section{Introduction}

Quantile regression describes how different parts of a response distribution vary with covariates. Since the check-loss formulation of \citet{koenker}, it has become a standard way to study conditional location, spread, and tail behavior. Auditing a fitted quantile forecast, however, requires care when the response distribution has point masses. For continuous distributions without an atom at the forecast, quantile validity gives the familiar nominal hit probability. For count data, rounded measurements, or outcomes with structural zeros, that equality need not hold. Consider a distribution assigning probabilities $(0.25,0.50,0.25)$ to $(-1,0,1)$. Its unique median is zero, yet $\Pp(Y\leq0)=0.75$.

A monitoring procedure based on exact-hit equality can therefore reject a correct discrete-outcome forecast. Pooling observations creates another difficulty: a forecast may be too low in one group and too high in another, with the errors cancelling in the aggregate. We address these problems by auditing the defining quantile inequalities conditionally on the information available before each response. The resulting procedure allows the evidence to be inspected repeatedly and the audit to stop when a threshold is crossed.

The main construction combines audits across the prunings of a fixed covariate tree. This lets the evidence use a coarse partition where little subdivision is needed and a finer partition where a departure is localized. Theorem~\ref{thm:tree} gives a finite-path bound that separates the information available in each retained cell from the costs of local betting and partition aggregation. Theorem~\ref{thm:universal} shows that mixing over increasingly fine trees recovers the full conditional sign-information rate under iid and positivity assumptions, and proves a matching upper bound. Proposition~\ref{prop:approx} measures the loss from using a finite partition. We also derive the local information projection and an exact mixture recurrence, and explain what the audit implies for quantile selection and regression inference.

\paragraph{Related work.}

Several strands of existing work provide the foundations for this audit. \citet{chernozhukov} construct quantile bands by inverting distribution-function bands, including for discrete outcomes, while \citet{howard} develop time-uniform quantile inference. Sequential testing through supermartingales is studied by \citet{casgrain} for elicitable and identifiable functionals and by \citet{waudby} for bounded means. The one-sided quantile e-statistics of \citet{wang} include the strict-upper-tail factor used here.

The constraints in our three-sign experiment fall within the complete-class theory of \citet{larsson} and its conditional extension by \citet{clerico}. Our calculation makes that class explicit for the quantile-set null. The continuous mixtures draw on universal portfolios \citep{cover,coverord}, and the tree recursion uses context-tree weighting \citep{willems}. Related contextual and countable-mixture constructions for sequential evidence appear in \citet{ramdasexchange}; \citet{orabona} connect universal-portfolio regret to time-uniform inference.

The closest recent forecast-auditing work is \citet{antonov}. Their feature-aware procedure tests exact conditional hit coverage and explicitly notes that this differs from quantile validity at atoms. They obtain finite-time guarantees for feature-aligned alternatives without an iid assumption. We instead test quantile-set membership and study the information available to a multiscale sign-based audit. This difference in targets matters when interpreting both the theoretical results and the numerical comparisons.

\paragraph{Contribution.}

Our contribution is the connection between the pruning bound, the attainable conditional information rate, and the loss caused by covariate coarsening for this two-inequality null. The pruning penalty is explicit: each retained cell pays its local mixture cost, and the tree pays $b_D(T)\log2$ for its stopping and splitting decisions. The rate result concerns iid alternatives with positive tie probability, whereas finite-sample validity follows from the conditional null. The paper develops a method for auditing issued forecasts; it does not propose a new quantile-regression estimator.

\section{The target and an identification limit}

For a distribution function $F$ and $0<\tau<1$, define
\begin{equation}\label{eq:qset}
\Q_\tau(F)=\{q\in\mathbb R:F(q-)\leq\tau\leq F(q)\}.
\end{equation}
When $\E|Y|<\infty$, this set also minimizes $\E\rho_\tau(Y-q)$, where $\rho_\tau(u)=u(\tau-\ind\{u<0\})$. The left and right derivatives of the convex objective are $F(q-)-\tau$ and $F(q)-\tau$, so its subgradient contains zero precisely when \eqref{eq:qset} holds. The audit uses the inequalities directly and needs no moment assumption on $Y$.

The lower generalized inverse, $q^-_\tau(F)=\inf\{y:F(y)\geq\tau\}$, selects one member of $\Q_\tau(F)$. When $F$ has a flat region at level $\tau$, the set can contain more than one value. The distinction between membership and selection then has an inferential consequence: without further restrictions on the distribution, a uniformly valid test of the lower-quantile convention cannot reject other members of the quantile set.

\begin{theorem}[Why the selection convention cannot be uniformly tested]
\label{thm:closure}
Fix a real $q$. Let $\mathcal H^-_q$ be all distributions whose lower $\tau$-quantile equals $q$, and let $\mathcal H^Q_q$ be all distributions for which $q\in\Q_\tau(F)$. If a possibly randomized test $\phi_n\in[0,1]$ based on $n$ independent observations satisfies
\[
\sup_{P\in\mathcal H^-_q}\E_{P^n}\phi_n\leq\alpha,
\]
then the same bound holds for every $P\in\mathcal H^Q_q$. Moreover, if an alarm time $T$ satisfies
\[
\sup_{P\in\mathcal H^-_q}\Pp_{P^\infty}(T<\infty)\leq\alpha,
\]
then $\Pp_{P^\infty}(T<\infty)\leq\alpha$ for every $P\in\mathcal H^Q_q$.
\end{theorem}

\begin{proof}
For $P\in\mathcal H^Q_q$ and $0<\varepsilon<1$, set $P_\varepsilon=(1-\varepsilon)P+\varepsilon\delta_q$. For every $y<q$, $F(y)\leq F(q-)\leq\tau$, so $F_\varepsilon(y)<\tau$. Also $F_\varepsilon(q)\geq\tau$. Hence $P_\varepsilon\in\mathcal H^-_q$. Couple each observation from $P_\varepsilon$ with one from $P$, replacing it by $q$ with probability $\varepsilon$. The probability of any replacement among $n$ observations is at most $n\varepsilon$, so
\[
|\E_{P_\varepsilon^n}\phi_n-\E_{P^n}\phi_n|\leq n\varepsilon.
\]
Let $\varepsilon\downarrow0$. For the sequential assertion apply the finite-sample argument to $\ind\{T\leq n\}$, and then let $n\uparrow\infty$.
\end{proof}

Theorem~\ref{thm:closure} concerns testing under an unrestricted null, rather than estimation of the lower quantile. Any test with uniform finite-sample type-I control for that selection has the same rejection bound at every other valid selection. Shape or separation assumptions can change this conclusion. The argument is the elementary closure observation that adding an arbitrarily small atom at $q$ makes $q$ the lower quantile.

\section{Conditional sampling and the three-outcome experiment}

Let $\F_{t-1}$ contain previous responses, predictors, and any initial training data. Let
\[
\G_t=\sigma(\F_{t-1},X_t),\qquad
\F_{t-1}\subseteq\G_t\subseteq\F_t,
\]
where $Y_t$ is observed after $\G_t$. The forecast $q_t$ and all current bets are $\G_t$-measurable. Covariates may be allocated using previous outcomes. The null is
\begin{equation}\label{eq:null}
\Pp(Y_t<q_t\mid\G_t)\leq\tau,\qquad
\Pp(Y_t>q_t\mid\G_t)\leq1-\tau
\quad\text{a.s., for every }t.
\end{equation}
Condition~\eqref{eq:null} evaluates the forecasts actually issued. An independently trained model makes its forecasts available before the audit responses, but their correctness remains the hypothesis to be tested.

Write $Z_t=L,T,R$ according as $Y_t<q_t$, $Y_t=q_t$, or $Y_t>q_t$. Its conditional probabilities $(a_t,b_t,c_t)$ belong under the null to
\begin{equation}\label{eq:poly}
\mathcal P_\tau=
\{(a,b,c)\geq0:a+b+c=1,\ a\leq\tau,\ c\leq1-\tau\}.
\end{equation}
The four vertices of $\mathcal P_\tau$ are $(0,1,0)$, $(\tau,1-\tau,0)$, $(0,\tau,1-\tau)$, and $(\tau,0,1-\tau)$. Every point in this set is an admissible distribution of the three signs. Requiring the hit probability to equal $\tau$ imposes an additional restriction, as the following example shows.

\begin{proposition}[False alarms caused by changing the null]
\label{prop:false}
Suppose $q_t=0$, $\tau=1/2$, and the responses are independent with probabilities $(1/4,1/2,1/4)$ at $(-1,0,1)$. Then \eqref{eq:null} holds. Nevertheless, for $H_t=\ind\{Y_t\leq0\}$,
\[
M_n^{\rm eq}=\prod_{t=1}^n(3/2)^{H_t}(1/2)^{1-H_t}
\]
exceeds every fixed threshold eventually with probability one.
\end{proposition}

\begin{proof}
The true median is zero. The strong law gives
\[
n^{-1}\log M_n^{\rm eq}
\to(3/4)\log(3/2)+(1/4)\log(1/2)>0.
\]
The displayed process is a valid exact-hit likelihood-ratio martingale when $\Pp(H_t=1\mid\G_t)=1/2$, but not an e-process for \eqref{eq:null}.
\end{proof}

\section{All valid sign factors and their information}

A sign factor is a nonnegative vector $e=(e_L,e_T,e_R)$. It is valid if $p^\top e\leq1$ for every $p\in\mathcal P_\tau$. For $p_0\in(0,1)$, a binary observation $h$, and $\lambda\in[0,1]$, define
\begin{equation}\label{eq:bet}
b_{p_0,\lambda}(h)=1+\lambda(h/p_0-1).
\end{equation}
Two valid families are
\[
e^L_\lambda(Z)=b_{\tau,\lambda}(\ind\{Z=L\}),\qquad
e^R_\lambda(Z)=b_{1-\tau,\lambda}(\ind\{Z=R\}).
\]
A tie receives the factor $1-\lambda$ in either family. Thus evidence for a strict-tail violation accumulates only when the corresponding strict-tail event occurs.

\begin{theorem}[Explicit complete class]\label{thm:complete}
A nonnegative sign factor $e$ is valid if and only if
\begin{equation}\label{eq:support}
e_T+\tau\pos{e_L-e_T}+(1-\tau)\pos{e_R-e_T}\leq1.
\end{equation}
Every valid factor is componentwise dominated by an element of
$\{e^L_\lambda:0\leq\lambda\leq1\}\cup\{e^R_\lambda:0\leq\lambda\leq1\}$.
Every element of this union is admissible among sign factors: no other valid sign factor strictly dominates it.
\end{theorem}

\begin{proof}
Because $(a,c)$ ranges independently over $[0,\tau]\times[0,1-\tau]$, maximizing
$e_T+a(e_L-e_T)+c(e_R-e_T)$ gives \eqref{eq:support}. Raise either tail value below $e_T$ to $e_T$; this does not change the maximum. Add the remaining nonnegative slack in \eqref{eq:support} to all three values. Denote the resulting factor by $u$. It dominates $e$, satisfies $u_L,u_R\geq u_T$, and obeys $\tau u_L+(1-\tau)u_R=1$.

Raise $u_T$ to $\min(u_L,u_R)$. Validity is preserved by \eqref{eq:support}. If $u_L\geq u_R$, the resulting vector is $(u_L,u_R,u_R)$ with $u_R\leq1$. Set $\lambda=1-u_R$ to obtain $e^L_\lambda$. The other case gives $e^R_\lambda$.

For admissibility consider $e^L_\lambda$. The null distribution $(\tau,0,1-\tau)$ gives expectation one and positive mass to both tails. A dominating valid vector must therefore agree on both tails. The null $(\tau,1-\tau,0)$ also gives expectation one, so the tie value cannot be increased. The argument for $e^R_\lambda$ uses $(0,\tau,1-\tau)$ instead.
\end{proof}

Theorem~\ref{thm:complete} reduces the one-step choice to two one-direction families. For any valid sign factor, one of these factors pays at least as much on every outcome. We will nevertheless mix entire evidence processes over directions, since the direction of a departure is unknown. The result applies to factors based on the three signs; residual magnitudes carry information outside this experiment.

Define binary divergence
\[
d(u\Vert v)=u\log\frac uv+(1-u)\log\frac{1-u}{1-v},
\]
with $0\log0=0$, and
\begin{equation}\label{eq:J}
J_\tau(a,b,c)=
\begin{cases}
d(a\Vert\tau),&a>\tau,\\
d(c\Vert1-\tau),&c>1-\tau,\\
0,&\text{otherwise}.
\end{cases}
\end{equation}
The two strict violations cannot occur together.

\begin{theorem}[Information projection and optimal log growth]
\label{thm:info}
For a strictly positive probability vector $p=(a,b,c)$,
\[
\sup_{e\text{ valid}}\E_p\log e(Z)
=\inf_{r\in\mathcal P_\tau}\KL(p\Vert r)=J_\tau(p).
\]
If $a>\tau$, a minimizing distribution and maximizing factor are
\[
r^*=\left(\tau,\frac{(1-\tau)b}{1-a},
\frac{(1-\tau)c}{1-a}\right),\quad
e^*=\left(\frac a\tau,\frac{1-a}{1-\tau},
\frac{1-a}{1-\tau}\right).
\]
The case $c>1-\tau$ is obtained by interchanging tails and replacing $\tau$ by $1-\tau$. If $p\in\mathcal P_\tau$, $r^*=p$ and $e^*=1$ are optimal. The value formula extends to boundary probabilities by continuity.
\end{theorem}

\begin{proof}
Suppose $a>\tau$. The displayed $r^*$ is feasible. For any feasible $r$, splitting the divergence into the event $L$ and its complement yields
\[
\KL(p\Vert r)=d(a\Vert r_L)+(1-a)
\KL\left(\frac{(b,c)}{1-a}\,\middle\Vert\,
\frac{(r_T,r_R)}{1-r_L}\right)\geq d(a\Vert\tau).
\]
The last inequality follows because $d(a\Vert v)$ decreases in $v<a$, and $r_L\leq\tau$. Equality holds at $r^*$. The displayed $e^*$ is $e^L_\lambda$ with $\lambda=(a-\tau)/(1-\tau)$, and has expected log factor $d(a\Vert\tau)$.
For any valid $e$, Jensen's inequality gives
\[
\E_p\log e-\KL(p\Vert r^*)
=\E_p\log\left(e\frac{r^*}{p}\right)
\leq\log\E_{r^*}e\leq0.
\]
This proves optimality. The upper-tail case is identical. Under the null, Jensen gives $\E_p\log e\leq0$, attained by $e=1$. Limits or direct restriction to the support of $p$ give the boundary value formula.
\end{proof}

For a known alternative, the optimal factor treats ties and the nonviolating tail alike. Its log-growth rate is therefore the divergence in the appropriate strict-tail Bernoulli experiment. Ties still matter for specifying that experiment correctly. For example, randomly assigning half the ties to each tail produces a lower-hit probability of $a+b/2$, which need not equal $\tau$.

\section{Sequential construction and contextual adaptation}

Choose the direction and stake using $\G_t$, before observing $Y_t$. Under \eqref{eq:null}, the resulting factor has conditional expectation at most one. Iterated conditioning then makes its cumulative product a nonnegative supermartingale relative to $\F_t$. This gives finite-sample validity even when responses are dependent, provided the conditional inequalities hold at every round.

To avoid choosing a direction and stake in advance, integrate the fixed-stake products. For $k$ successes among $m$ Bernoulli observations, set
\begin{equation}\label{eq:B}
B_{p_0}(k,m)=\int_0^1
\{1+\lambda(1/p_0-1)\}^{k}(1-\lambda)^{m-k}\,d\lambda.
\end{equation}
If $N_{L,n}$ and $N_{R,n}$ are the strict-tail counts, define
\begin{equation}\label{eq:U}
U_n=\tfrac12 B_\tau(N_{L,n},n)
+\tfrac12 B_{1-\tau}(N_{R,n},n),\qquad U_0=1.
\end{equation}
Equation~\eqref{eq:U} averages the cumulative evidence from the two directions after integrating over stakes. The order matters: the products are formed before they are averaged.

Let $G_t=g(X_t)\in\{1,\ldots,K\}$ be a group label known before $Y_t$. Within each group build the process in \eqref{eq:U} from only its observations, writing its current value as $U_{g,n}$. Set
\begin{equation}\label{eq:C}
C_n=\prod_{g=1}^K U_{g,n}.
\end{equation}
At each round, only the factor for the current group changes. This makes the disjoint-group product valid without requiring independence between groups. To allow for a partition that adds little useful information, we can also retain the pooled process by setting $A_n=(U_n+C_n)/2$. More generally, a fixed probability mixture can combine audits for several partitions.

\begin{theorem}[Validity and finite-path evidence bound]
\label{thm:valid}
Under \eqref{eq:null}, $U_n$, $C_n$, and $A_n$ are nonnegative supermartingales starting at one. Thus each satisfies
\[
\Pp\{\exists n\geq1:E_n\geq1/\alpha\}\leq\alpha.
\]
For every realized sequence, let $n_g$ be its group sizes and $\widehat p_g$ its three empirical sign proportions. Then
\begin{equation}\label{eq:regret}
\log C_n\geq\sum_{g:n_g>0} n_gJ_\tau(\widehat p_g)
-\sum_{g:n_g>0}\log\{2(n_g+1)\}.
\end{equation}
In particular, if the right-hand side exceeds $\log(1/\alpha)$, the contextual audit has crossed its rejection threshold by time $n$.
\end{theorem}

\begin{proof}
Each fixed-stake product has conditional expectation at most its previous value. Integrating and mixing preserves this property. Every group mixture is strictly positive at finite times because stakes in a neighborhood of zero have positive prior mass. Its update ratio is a posterior-weighted mean of valid one-step factors. In \eqref{eq:C}, only the current group's ratio changes, so conditional expectation of the product is also at most its previous value. The maximum inequality for a nonnegative supermartingale gives the displayed error bound.

For the pathwise inequality, write a product of $m$ binary factors as
$W(\lambda)=\prod_{i=1}^m\{(1-\lambda)+\lambda x_i\}$, with $x_i\geq0$. Expansion gives
$W(\lambda)=\sum_{j=0}^m c_j\lambda^j(1-\lambda)^{m-j}$ with $c_j\geq0$. Since
\[
\int_0^1\lambda^j(1-\lambda)^{m-j}d\lambda
=\frac{1}{(m+1)\binom mj},\quad
\binom mj\lambda^j(1-\lambda)^{m-j}\leq1,
\]
$W(\lambda)\leq(m+1)\int_0^1W(v)dv$ for all $\lambda$.
For a binary success proportion $\widehat u$, maximizing the log product in \eqref{eq:B} over $\lambda$ gives $m d(\widehat u\Vert p_0)$ if $\widehat u>p_0$ and zero otherwise. The two-direction mixture loses at most $\log2$ relative to its better component. Apply this bound within each nonempty group and add.
\end{proof}

\begin{corollary}[Contextual information rate]\label{cor:rate}
Suppose $(G_t,Z_t)$ are independent and identically distributed, $\Pp(G_t=g)=\pi_g>0$, and $Z_t\mid G_t=g$ has distribution $p_g$. Then
\[
\frac1n\log C_n\longrightarrow I=\sum_g\pi_gJ_\tau(p_g)
\quad\text{almost surely}.
\]
The pooled process has rate $J_\tau(\sum_g\pi_gp_g)$. In particular $C_n$ eventually rejects almost surely whenever $I>0$.
\end{corollary}

\begin{proof}
The lower bound follows from \eqref{eq:regret}, the strong law, continuity of $J_\tau$, and fixed $K$. Conversely, each integral in \eqref{eq:B} is bounded above by its maximum integrand. Thus $U_{g,n}\leq\exp\{n_gJ_\tau(\widehat p_g)\}$. Apply the strong law again. The same argument applies to the pooled process.
\end{proof}

The iid assumption in Corollary~\ref{cor:rate} is needed for the rate calculation, while Theorem~\ref{thm:valid} allows the broader conditional sampling scheme. The gain from using groups can be substantial. At $\tau=1/2$, take two equally likely groups with $p_1=(0.65,0.10,0.25)$ and $p_2=(0.25,0.10,0.65)$. Pooling gives $(0.45,0.10,0.45)$, which satisfies the quantile null and has information rate zero. The grouped audit instead has rate $I=d(0.65\Vert0.5)>0$. A finer partition can reveal such cancellation, although its mixture penalties and smaller cell counts may delay an alarm.

\paragraph{An information lower bound.}

For strictly positive $p_g$, let $P^*$ preserve the group probabilities and replace each $p_g$ by its null projection from Theorem~\ref{thm:info}. Then the per-observation divergence is $I$. Any alarm time $T$ based only on the group--sign observations in this iid experiment, with $P^*(T<\infty)\leq\alpha$, $Q(T<\infty)=1$, and $\E_QT<\infty$ satisfies
\begin{equation}\label{eq:stop}
\E_QT\geq\frac{\log(1/\alpha)}{I},\qquad I>0.
\end{equation}
Wald's identity gives expected stopped log-likelihood ratio $I\E_QT$, since the increments are bounded on the finite alphabet with strictly positive probabilities. Applying data processing to the event of a finite alarm gives a stopped-experiment divergence of at least $\log(1/P^*(T<\infty))$, which proves \eqref{eq:stop}. This is a lower bound on mean alarm time. Attaining an almost-sure information rate does not, on its own, establish attainment of this stopping-time bound.

\section{Multiscale contextual auditing}\label{sec:tree}

A fixed partition involves a tradeoff. Coarse cells can hide localized departures, while a fine partition can spend evidence on subdivisions that contribute little signal. We use context-tree weighting \citep{willems} to combine partitions before seeing the audit responses. The mixture interpretation follows the contextual and countable-mixture approach to sequential evidence in \citet{ramdasexchange}.

Let $\mathcal T_D$ be a fixed complete binary covariate tree of depth $D$. Each node $v$ represents a measurable cell, and its children partition that cell. The cells may be selected using independent training data, but are fixed during the audit. Let $U_{v,n}$ be \eqref{eq:U} using the signs of observations whose covariates fall in cell $v$. Set
\begin{equation}\label{eq:tree}
W_{v,n}=\begin{cases}
U_{v,n},& |v|=D,\\
\tfrac12 U_{v,n}+\tfrac12 W_{v0,n}W_{v1,n},& |v|<D.
\end{cases}
\end{equation}
Empty cells have $U_{v,n}=1$. A pruning $T$ stops at some internal nodes and splits at the others, so its leaf cells form a partition. Write $b_D(T)$ for the number of visited nodes at depth less than $D$, counting both stopping and splitting decisions. Forced terminal leaves cost no decision. If $T$ has $K$ leaves, then $b_D(T)\leq2K-1$.

\begin{theorem}[Tree oracle inequality]\label{thm:tree}
Under \eqref{eq:null}, the root process $W^{(D)}_n$ in \eqref{eq:tree} is a nonnegative supermartingale starting at one. For every realized sequence and every pruning $T$,
\begin{align}\label{eq:treebound}
\log W^{(D)}_n\ \geq\ &
\sum_{v\in\operatorname{leaves}(T):n_v>0}
n_vJ_\tau(\widehat p_v)\notag\\
&-\sum_{v\in\operatorname{leaves}(T):n_v>0}
\log\{2(n_v+1)\}-b_D(T)\log2.
\end{align}
In particular, the oracle comparison holds simultaneously over all prunings without selecting a partition from the responses.
\end{theorem}

\begin{proof}
Expand the recursion. It gives the exact identity
\[
W^{(D)}_n=\sum_{T}\pi_D(T)
\prod_{v\in\operatorname{leaves}(T)}U_{v,n},
\qquad \pi_D(T)=2^{-b_D(T)}.
\]
The prior sums to one: at each nonterminal node it stops with probability $1/2$ or splits with probability $1/2$, with forced stopping at depth $D$. For a fixed pruning, exactly one leaf receives an observation at a time. Its updated mixture capital has conditional expected ratio at most one, given the past and the current covariate, by Theorem~\ref{thm:valid}. Multiplying by the unchanged capitals preserves that inequality. No independence between cells is required. A fixed convex combination of these supermartingales is again a supermartingale. Retaining the single term for $T$ and applying \eqref{eq:regret} proves the bound. Zero capitals may be handled directly by the component mixture representation; the continuous mixtures here are strictly positive for finite samples.
\end{proof}

Although \eqref{eq:tree} uses evidence at several levels of the tree, its expansion is a mixture of disjoint-cell products. Each observation contributes once within any given pruning. Multiplying the evidence from overlapping cells directly would reuse observations and generally lose validity. Only the ancestors of the current covariate cell need updating, so each observation requires $O(D)$ tree updates once local mixture values are available.

\subsection{Recovering the full conditional information}

Consider now an iid sign experiment $(X_t,Z_t)$ with $X_t\in[0,1]^d$ and conditional probabilities $p(x)=(a(x),b(x),c(x))$. It is induced, for example, by a fixed deployed forecast $q(x)$ and iid pairs $(X_t,Y_t)$. Bisect coordinates cyclically to obtain nested dyadic trees. Their finite partitions generate the Borel sigma field. No density of $X$ is assumed; use half-open cells with the outer boundary included. Define
\[
V_n=\sum_{D=0}^{\infty}2^{-D-1}W^{(D)}_n,
\qquad I=\E J_\tau(p(X)).
\]
The countable mixture lets every fixed resolution retain positive prior weight. It provides the benchmark for the rate theorem below; the numerical implementation uses a finite depth.

\begin{theorem}[Conditional information attainment]
\label{thm:universal}
The process $V$ is valid under \eqref{eq:null}. Suppose that under an iid alternative $Q$ there exists $\eta>0$ such that $p_z(X)\geq\eta$ almost surely for each $z\in\{L,T,R\}$. Then
\begin{equation}\label{eq:universal}
\frac{1}{n}\log V_n\longrightarrow I
\qquad Q\text{-almost surely}.
\end{equation}
Moreover, if $E_n\geq0$ is any evidence statistic measurable with respect to the first $n$ sign--covariate pairs, with $\E_P E_n\leq1$ for every iid law satisfying the conditional quantile null, then
\[
\limsup_{n\to\infty}\frac{\log E_n}{n}\leq I
\qquad Q\text{-almost surely}.
\]
Thus the bound concerns all such valid sign-based evidence sequences, not only tree mixtures. If the conditional quantile null is violated on a set of positive covariate probability, then $I>0$ and $V$ eventually crosses any fixed threshold almost surely.
\end{theorem}

\begin{proof}
Validity follows from the nonnegative countable mixture and conditional monotone convergence. For fixed $D$, the full depth-$D$ partition is one pruning with positive weight. The strong law and Corollary~\ref{cor:rate} give
\[
\liminf_n n^{-1}\log V_n\geq I_D
:=\E J_\tau\{\E[p(X)\mid\mathcal A_D]\},
\]
where $\mathcal A_D$ is the sigma field of its cells. Cells of zero $Q$-probability receive no observations almost surely and may be omitted when applying Corollary~\ref{cor:rate}. These statements hold simultaneously for the countable set of depths. The conditional-expectation martingale convergence theorem gives $\E[p(X)\mid\mathcal A_D]\to p(X)$ almost surely. The function $J_\tau$ is continuous and bounded on the probability simplex, so dominated convergence yields $I_D\to I$. This proves the lower bound.

For the upper bound, let $r^*(x)$ be the null information projection of $p(x)$ from Theorem~\ref{thm:info}, and let $P^*$ have the same covariate marginal as $Q$ and conditional signs $r^*(x)$. This is a measurable iid null law. Under the positivity assumption, its sign probabilities are bounded away from zero by a constant depending only on $\eta$ and $\tau$. Indeed, on the left-violating region the nonleft coordinates are $(1-\tau)p_z/(1-a)\geq(1-\tau)\eta$ and the left coordinate is $\tau$; the other region is analogous. Thus $Q$ and $P^*$ are equivalent on each finite sample. Let
\[
L_n=\prod_{t=1}^n\frac{p_{Z_t}(X_t)}{r^*_{Z_t}(X_t)}.
\]
The logarithms are bounded and iid, so $n^{-1}\log L_n\to I$ almost surely. For every $\epsilon>0$,
\[
Q(E_n/L_n\geq e^{n\epsilon})
\leq e^{-n\epsilon}\E_Q(E_n/L_n)
=e^{-n\epsilon}\E_{P^*}E_n\leq e^{-n\epsilon}.
\]
Borel--Cantelli, followed by a countable sequence $\epsilon\downarrow0$, proves the upper bound. Apply it to $V_n$ to obtain equality. Finally, $J_\tau(p)>0$ exactly outside $\mathcal P_\tau$, proving the consistency assertion.
\end{proof}

The positivity assumption includes the tie probability. Thus Theorem~\ref{thm:universal} establishes the rate for iid alternatives with all three sign probabilities uniformly bounded away from zero. Continuous outcomes with no ties, drifting alternatives, and finite-depth truncations fall outside this rate statement. The finite-sample validity of the mixtures continues to follow from \eqref{eq:null} in those settings.

\begin{proposition}[Information lost at finite resolution]
\label{prop:approx}
Under the positivity assumption of Theorem~\ref{thm:universal}, for any finite covariate partition $\mathcal A$, put $\bar a=\E(a\mid\mathcal A)$ and $\bar c=\E(c\mid\mathcal A)$. Then
\[
0\leq I-\E J_\tau(\E[p\mid\mathcal A])
\leq
\frac{\E\{(a-\bar a)^2+(c-\bar c)^2\}}{2\eta(1-\eta)}.
\]
If both strict-tail probabilities are H\"older of order $s\in(0,1]$ with constant $H$ in Euclidean distance and every cell has diameter at most $h$, the upper bound is at most $H^2h^{2s}/\{\eta(1-\eta)\}$.
\end{proposition}

\begin{proof}
Write $f_u(v)=d(v\Vert u)\ind\{v>u\}$. On the simplex, $J_\tau(p)=f_\tau(a)+f_{1-\tau}(c)$, since the two violations cannot coexist. Each $f_u$ is convex, continuously differentiable, and has Lipschitz derivative at most $1/\{\eta(1-\eta)\}$ on $[\eta,1-\eta]$: its second derivative is either zero or $1/\{v(1-v)\}$, with matching first derivative zero at $v=u$. The smooth convex Taylor bound applied conditionally about $\bar a,\bar c$ has vanishing linear terms after taking expectations. Jensen's inequality gives the lower bound. Within a cell the oscillation of each tail probability is at most $Hh^s$, so each squared centered deviation is bounded by $H^2h^{2s}$. Substitution proves the final claim.
\end{proof}

\subsection{Exact continuous mixtures without a stake grid}

For a horizon $N$, all local mixture values can be tabulated stably. For $0<p<1$ and $0\leq k\leq m$, the integral in \eqref{eq:B} obeys
\begin{equation}\label{eq:recurrence}
B_p(0,m)=\frac{1}{m+1},\qquad
B_p(k,m)=
\frac{1+k(1-p)B_p(k-1,m)/p}{m+1-k}
\quad(k\geq1).
\end{equation}
For the derivation, substitute $r=p+(1-p)\lambda$ in \eqref{eq:B} and integrate the derivative of $r^k(1-r)^{m-k+1}$ over $[p,1]$. The boundary term is $-p^k(1-p)^{m-k+1}$. Rearranging and normalizing gives the constant term one and the remaining term involving $B_p(k-1,m)$ in \eqref{eq:recurrence}.

We evaluate the recurrence with log-sum-exp to avoid overflow. For each distinct tail bound, tabulation through a horizon $N$ takes $O(N^2)$ time and memory. Subsequent local updates are table lookups, giving an $O(D)$ tree update per observation. The recurrence evaluates the continuous mixture in floating-point arithmetic and avoids choosing a finite stake grid. Its quadratic storage requirement limits the horizon; it is not a constant-memory implementation of the infinite mixture $V$.

\section{Regression inference and implementation}

For a candidate parameter $\beta$, set $q_t(\beta)=X_t^\top\beta$ and construct $E_n(\beta)$ by the preceding procedure. Define
\begin{equation}\label{eq:cs}
\mathcal C_n=\{\beta:\max_{s\leq n}E_s(\beta)<1/\alpha\}.
\end{equation}
For every fixed $\beta_0$ satisfying \eqref{eq:null} at all times,
\[
\Pp\{\beta_0\in\mathcal C_n\text{ for all }n\}\geq1-\alpha.
\]
Apply Theorem~\ref{thm:valid} at the fixed true value $\beta_0$ to obtain this coverage statement. Inverting the tests does not require a multiplicity correction across candidate parameters for pointwise coverage of $\beta_0$. Projection gives a confidence sequence for a coordinate or fixed linear contrast. Simultaneous containment of an entire partially identified parameter set is a different requirement.

The coverage claim requires conditional quantile correctness. A coefficient minimizing unconditional expected check loss in a misspecified linear family need not satisfy \eqref{eq:null}, and the audit may reject it. Computation also requires care: the signs change across the hyperplanes $Y_t=X_t^\top\beta$, and exact-fit faces have different tie factors from neighboring regions. A coarse grid is therefore only an exploratory approximation to the continuous parameter set. The numerical examples below audit issued forecasts rather than compute coefficient confidence sets.

\paragraph{Finite-grid comparator.}

For the auxiliary experiment and count-response illustration, we use $L=20$ fixed stakes $\lambda_\ell=(\ell-1/2)/L$, $\ell=1,\ldots,L$, and replace each integral by the uniform average of its $L$ products. This is a valid discrete mixture, not a quadrature approximation on which validity depends. If $M_{g,d,\ell}$ is the capital in group $g$, direction $d$, and stake $\ell$, initialize all capitals at one and update the current group's capitals by
\[
M_{g,d,\ell}\leftarrow M_{g,d,\ell}
\left[1+\lambda_\ell
\left\{\frac{\ind\{Z_t=d\}}{p_d}-1\right\}\right],
\quad p_L=\tau,\ p_R=1-\tau.
\]
The local capital is $U_{g,n}=(2L)^{-1}\sum_{d,\ell}M_{g,d,\ell}$. We store log capitals and use log-sum-exp to avoid underflow. Updating a partition costs $O(L)$ per observation and requires $O(KL)$ storage. The grid mixture pays at most $\log(2L)$ relative to its best component, but its finite set of positive stakes can miss sufficiently weak departures. It therefore does not inherit \eqref{eq:regret} or consistency for arbitrarily small violations. A countable dense prior restores consistency, while the continuous mixture gives the stated finite-path bound.

\section{Reproducible numerical study}

\subsection{Adaptation across covariate resolutions}

The primary experiment uses the exact continuous mixtures of \eqref{eq:recurrence}. For each of eight designs we simulate 2000 independent streams of length 1500, with $\tau=1/2$, $\alpha=0.05$, and PCG64 seed 20261012. Let $X\in[0,1]$ and issue $q(X)=1+2X$ for $Y=q(X)+V$, where $V\in\{-1,0,1\}$ has probabilities $(a,b,c)$. Apart from the adaptive null, $X$ is uniform and independent across observations. We compare the pooled audit, fixed dyadic partitions of depths 1, 3, and 5, and the depth-5 tree mixture. Each procedure uses the same local continuous mixtures and tests the same conditional quantile-set null.

The continuous-null sign law is $(.5,0,.5)$, which also arises from a continuous symmetric residual; the atomic null is $(.25,.5,.25)$. For the adaptive null, $(a,b,c)=(.5,.2,.3)$ on $X<.5$ and $(.3,.2,.5)$ otherwise. Its covariate falls in the upper half with probability $.8$ after a left sign and $.2$ otherwise, uniformly within the selected half, with initial preceding sign left. The global alternative has $(.6,.1,.3)$ everywhere. Coarse cancellation uses $(.65,.1,.25)$ on the lower half and $(.25,.1,.65)$ on the upper half. The localized alternative uses $(.8,.1,.1)$ on $[0,1/8)$ and $(.4,.2,.4)$ elsewhere. Fine cancellation alternates $(.7,.1,.2)$ and $(.2,.1,.7)$ across the eight depth-3 cells. Weak cancellation uses $(.55,.1,.35)$ and its reversal on the two halves.

\begin{table}[ht]
\centering\small
\caption{Rejection at any time through 1500 observations, 2000 streams per design. All columns test the same quantile-set null. Monte Carlo standard errors are at most $0.0112$; exact binomial intervals and alarm times are supplied in the supplement.}
\label{tab:tree}
\begin{tabular}{lccccc}\toprule
Design & Pooled & Depth 1 & Depth 3 & Depth 5 & Tree\\\midrule
Continuous null & 0.0395 & 0.0445 & 0.0430 & 0.0360 & 0.0415\\
Atomic null & 0.0000 & 0.0000 & 0.0000 & 0.0000 & 0.0000\\
Adaptive null & 0.0010 & 0.0170 & 0.0005 & 0.0000 & 0.0030\\
Global alternative & 1.0000 & 1.0000 & 0.9595 & 0.0725 & 1.0000\\
Coarse cancellation & 0.0105 & 1.0000 & 1.0000 & 0.9860 & 1.0000\\
Localized error & 0.0040 & 0.0060 & 0.7965 & 0.0005 & 0.9710\\
Fine cancellation & 0.0060 & 0.0055 & 1.0000 & 1.0000 & 1.0000\\
Weak cancellation & 0.0060 & 0.5360 & 0.0435 & 0.0025 & 0.3455\\
\bottomrule\end{tabular}
\end{table}

The tree adapts to several forms of departure in Table~\ref{tab:tree}. For the localized error, it rejects with probability $0.971$, compared with $0.7965$ for the full depth-3 partition. A pruning can refine the affected region while leaving other cells coarse. For a global departure, the tree and pooled audit reject every stream, whereas the full depth-5 partition rejects only $0.0725$. Adaptation also has a cost: under weak cancellation, the correctly chosen depth-1 partition rejects $0.536$ of streams, compared with $0.3455$ for the tree. The null rejection proportions are consistent with the theoretical bound. These comparisons describe how partition choice affects this audit; they do not compare every possible feature-learning strategy.

We checked the recurrence against numerical quadrature in 300 cases, with maximum relative error $9.0\times10^{-15}$. Explicit enumeration of the five depth-2 prunings agreed with the tree recursion, and the finite-path oracle bound held on 200 generated streams. These numerical checks provide an implementation check alongside the proofs.

\subsection{Matched-null predictable betting and detection times}
\label{sec:matched}

We next compare the tree with three predictable betting rules on the same 2000 streams per design. Each uses the strict-tail factors in \eqref{eq:bet}, the two-inequality null, and threshold 20. This keeps the target fixed while varying the local betting rule and the way covariate information is used. The stake rules were specified before the comparison was run.

The pooled plug-in comparator estimates each strict-tail probability from previous signs. If $s$ successes have occurred in $m$ previous observations, it uses
\[
\widehat u=(s+1)/(m+2),\qquad
\lambda=\min\{0.99,\max\{0,(\widehat u-p_d)/(1-p_d)\}\},
\quad p_L=\tau,\ p_R=1-\tau.
\]
The pooled plug-in capital averages the left- and right-direction products. The cell plug-in rule estimates the probabilities separately in the eight depth-3 cells, mixes directions within each cell, and multiplies the cell capitals as in \eqref{eq:C}. These are predictable bounded-mean bets of the kind studied by \citet{waudby}; the plug-in stake follows from the oracle factor in Theorem~\ref{thm:info}.

The feature comparator uses the same eight cell indicators, with a separate Online Newton Step learner for each direction. We adapt the contextual betting approach of \citet{antonov} to the strict-tail inequalities. Initialize $\theta_{d,j}=0$ and $A_{d,j}=1$. In the current cell $j$, issue $\lambda=\theta_{d,j}$, observe $v=\ind\{Z_t=d\}/p_d-1$, multiply the direction's capital by $1+\lambda v$, and update
\[
g=\frac{v}{1+\theta_{d,j}v},\qquad
A_{d,j}\leftarrow A_{d,j}+g^2,\qquad
\theta_{d,j}\leftarrow
\Pi_{[0,0.99]}(\theta_{d,j}+g/A_{d,j}).
\]
Only the coordinate for the current cell is updated. The reported evidence averages the two direction capitals. All stakes are chosen before the response and lie in $[0,0.99]$, so the conditional mean factor is at most one under \eqref{eq:null}. Disjoint-cell products and direction mixtures preserve this validity. This comparator changes the target of \citet{antonov} to the strict-tail inequalities; its fixed dictionary, initialization, and update rule define the particular configuration evaluated here.

\begin{table}[ht]
\centering\small
\caption{Matched-null anytime rejection proportions on the primary streams. Plug-in columns use pooled signs or eight depth-3 cells; feature ONS uses those eight cell indicators. The tree column repeats Table~\ref{tab:tree}. Exact binomial intervals, paired rejection differences, and all alarm times are supplied with the reproducibility files.}
\label{tab:matched}
\begin{tabular}{lcccc}\toprule
Design & Pooled plug-in & Cell plug-in & Feature ONS & Tree\\\midrule
Continuous null & 0.0320 & 0.0420 & 0.0415 & 0.0415\\
Atomic null & 0.0000 & 0.0000 & 0.0000 & 0.0000\\
Adaptive null & 0.0005 & 0.0210 & 0.0090 & 0.0030\\
Global alternative & 1.0000 & 0.9955 & 0.1830 & 1.0000\\
Coarse cancellation & 0.0060 & 1.0000 & 0.4940 & 1.0000\\
Localized error & 0.0020 & 1.0000 & 0.2840 & 0.9710\\
Fine cancellation & 0.0015 & 1.0000 & 0.6220 & 1.0000\\
Weak cancellation & 0.0040 & 0.3000 & 0.0275 & 0.3455\\
\bottomrule\end{tabular}
\end{table}

We report the restricted mean alarm time $\E\min(T,1500)$, assigning 1500 to a stream that has not alarmed by the horizon. Unlike a median among rejections, this measure uses every stream. Table~\ref{tab:times} reports the empirical means for the alternatives. Smaller values indicate earlier detection within the horizon and should be read together with the rejection proportions.

\begin{table}[ht]
\centering\small
\caption{Empirical restricted mean alarm times $\widehat\E\min(T,1500)$ for the matched-null comparison. Each entry averages all 2000 streams, including nonrejections; Monte Carlo standard errors are supplied with the reproducibility files.}
\label{tab:times}
\begin{tabular}{lcccc}\toprule
Design & Pooled plug-in & Cell plug-in & Feature ONS & Tree\\\midrule
Global alternative & 256.3 & 464.8 & 1289.5 & 273.7\\
Coarse cancellation & 1491.1 & 173.3 & 976.8 & 157.9\\
Localized error & 1497.0 & 305.7 & 1271.1 & 841.5\\
Fine cancellation & 1497.8 & 93.7 & 700.3 & 208.9\\
Weak cancellation & 1494.0 & 1273.7 & 1462.5 & 1295.4\\
\bottomrule\end{tabular}
\end{table}

The cell plug-in rule detects localized errors much earlier than the tree: the restricted mean times are 305.7 and 841.5, with rejection proportions 1.0000 and 0.9710. It is also faster under fine cancellation. The tree is faster for the global departure and rejects more often under weak cancellation, although its restricted mean time is larger in the latter case. The feature comparator performs less well in this configuration, but its results depend on the chosen dictionary and learner settings. The practical benefit of the tree is adaptation over resolutions with a pathwise guarantee; these experiments also show that a well-chosen local betting rule can be faster for a particular alternative.

\subsection{Auxiliary comparison of quantile and exact-hit targets}

We ran $R=2000$ independent repetitions per design, each with $N=1000$ observations, $\tau=0.5$, and threshold $1/\alpha=20$. A NumPy PCG64 generator used seed 20261011. Each response has the form
\[
Y_t=1+2X_t+V_t,\qquad q_t=1+2X_t,\qquad X_t\in\{0,1\}.
\]
Table~\ref{tab:design} specifies the sign probabilities of $V_t$. Nonzero magnitudes may be set to one; independent positive exponential magnitudes would give the same sign experiment in the continuous-null design. The implemented audit uses only signs. Outside the adaptive design, $X_t$ is an independent fair Bernoulli variable. In the adaptive design, its conditional success probability is $0.8$ after a left sign and $0.2$ otherwise, with the preceding sign initialized to $L$.

\begin{table}[ht]
\centering\small
\caption{Simulation sign probabilities $(a,b,c)$. The first four designs satisfy the conditional quantile null.}
\label{tab:design}
\begin{tabular}{lll}\toprule
Design & $X=0$ & $X=1$\\\midrule
Continuous null & $(.50,0,.50)$ & $(.50,0,.50)$\\
Atom interior & $(.25,.50,.25)$ & $(.25,.50,.25)$\\
Atom boundary & $(.50,.25,.25)$ & $(.50,.25,.25)$\\
Adaptive null & $(.50,.20,.30)$ & $(.30,.20,.50)$\\
Global alternative & $(.65,.10,.25)$ & $(.65,.10,.25)$\\
Cancellation & $(.65,.10,.25)$ & $(.25,.10,.65)$\\
Weak cancellation & $(.55,.10,.35)$ & $(.35,.10,.55)$\\
\bottomrule
\end{tabular}
\end{table}

We compare the pooled discrete mixture, the product of the two group mixtures, their equal mixture, and a pooled exact-hit audit. The last uses the same stakes and two betting directions, but tests a Bernoulli variable $\ind\{Y_t\leq q_t\}$ with null mean $\tau$. Its target is therefore different. Rejection at a valid atomic quantile illustrates this distinction; it is not a type-I error relative to the exact-hit equality null. The comparator is a pooled equality audit, rather than the full feature-learning procedure of \citet{antonov}.

We record whether each process crosses the threshold at any time through $N$ and report Monte Carlo uncertainty for the estimated rejection proportion.

\begin{table}[ht]
\centering\small
\caption{Anytime rejection proportions through 1000 observations. Parentheses show Monte Carlo standard errors. For proportions zero or one, the exact two-sided 95\% binomial intervals are $[0,0.00184]$ and $[0.99816,1]$, respectively.}
\label{tab:sim}
\begin{tabular}{lcccc}\toprule
Design & Pooled & Context & Combined & Exact hit\\\midrule
Continuous null & 0.0325 (0.004) & 0.0400 (0.004) & 0.0365 (0.004) & 0.0325 (0.004)\\
Atom interior & 0.0000 (0.000) & 0.0000 (0.000) & 0.0000 (0.000) & 1.0000 (0.000)\\
Atom boundary & 0.0155 (0.003) & 0.0085 (0.002) & 0.0130 (0.003) & 1.0000 (0.000)\\
Adaptive null & 0.0000 (0.000) & 0.0115 (0.002) & 0.0050 (0.002) & 1.0000 (0.000)\\
Global alternative & 1.0000 (0.000) & 1.0000 (0.000) & 1.0000 (0.000) & 1.0000 (0.000)\\
Cancellation & 0.0060 (0.002) & 1.0000 (0.000) & 1.0000 (0.000) & 0.5180 (0.011)\\
Weak cancellation & 0.0110 (0.002) & 0.3280 (0.010) & 0.2680 (0.010) & 0.5360 (0.011)\\
\bottomrule\end{tabular}
\end{table}

The quantile-set audits reject fewer than $0.05$ of streams in each of the four null designs. At the interior atom, none rejects, while the exact-hit audit rejects every stream. In the cancellation design, the contextual audit detects every stream with median alarm time 109; the pooled rejection proportion is only 0.006. Pooling is faster when both groups have the same directional error, with median alarm times 86 and 105 for pooled and contextual audits. Under weak cancellation, contextual power is only 0.328 by 1000 observations. The examples show both the value of using groups and the cost of dividing the evidence among them.

\subsection{A held-out count-response illustration}

The \texttt{DoctorVisits} dataset contains 5190 observations and is distributed with the AER package \citep{aer}. The response counts doctor visits during the preceding two weeks. A permutation with seed 314159 assigns 2595 observations to training and 2595 to the audit. This ordering creates a cross-sectional holdout stream; it does not follow the records chronologically.

At each of $\tau=0.5$ and $0.9$, we fit ordinary linear check-loss regression by linear programming, using an intercept, age, income, illness count, health score, and an indicator for female gender, in the data's documented units. We issue the nonnegative integer forecast
\[
q(x)=\max\{0,\lfloor x^\top\widehat\beta+1/2\rfloor\}.
\]
The audit evaluates this rounded forecast directly; rounding a fitted linear quantile does not itself guarantee quantile correctness. Before examining the held-out responses, we partition the audit into subjects with zero illness count and those with a positive count. We use the same finite stakes as in the auxiliary simulation and average the pooled and contextual evidence.

\begin{table}[ht]
\centering\small
\caption{Held-out doctor-visit results. Tail columns are empirical proportions. ``None'' means no threshold crossing in 2595 audit observations.}
\label{tab:real}
\begin{tabular}{cccccc}\toprule
$\tau$ & $Y<q$ & $Y=q$ & $Y>q$ & Quantile-set alarm & Exact-hit alarm\\\midrule
0.5 & 0.000 & 0.802 & 0.198 & None & 28\\
0.9 & 0.644 & 0.311 & 0.045 & None & 195\\
\bottomrule\end{tabular}
\end{table}

The fitted median forecast is zero for every subject, and $80.2\%$ of held-out responses equal that forecast. Exact-hit equality therefore gives a different assessment from quantile-set membership. At both quantile levels, the equality audit rejects while the quantile-set audit does not. For $\tau=0.9$, forecasts range from zero to three, and both pooled strict-tail frequencies lie below their null bounds. The same holds in the two illness groups. Nonrejection leaves open the possibility of departures in finer covariate groups.

We also apply the continuous-mixture audit using a depth-3 tree chosen from training predictors alone. The root separates zero from positive illness counts. The next level splits age at each training cell's median, and the final level does the same for income. Empty training cells use cutoff zero, and left children include values equal to the cutoff. Neither forecast is rejected during the 2595 held-out observations.

To examine sensitivity to a localized error, we increase the forecast by one visit in the zero-illness group and keep the tree and holdout order fixed. The tree alarms at observations 30 and 935 for $\tau=0.5$ and $0.9$. In contrast, the pooled strict-tail proportions remain within their null bounds: $(0.280,0.181)$ for the median and $(0.838,0.037)$ for the 0.9-quantile. The pooled continuous-mixture audit does not alarm. This deliberately introduced departure shows how subgroup evidence can reveal an error hidden by pooling; it does not establish an error in the original fitted forecasts.

The statistical interpretation depends on the sampling assumptions. These historical survey records do not provide an independently randomized sequential trial or a design-based survey guarantee. Nor can a nonrejection establish conditional correctness at every covariate value. The audit concerns the issued forecast and the information used to assess it, rather than causal effects of illness, income, or gender. We report each quantile separately at level $0.05$; joint familywise testing would require an adjustment or an across-quantile mixture.

\section{Discussion}

Auditing discrete quantile forecasts starts with choosing the right target. The strict-tail inequalities retain the usual quantile definition at point masses, while the information projection describes the evidence available from residual signs. The tree audit then balances this conditional signal against the costs of local mixtures and partition refinement. Under the assumptions of Theorem~\ref{thm:universal}, refining-tree mixtures attain the largest almost-sure exponential rate available to valid sign-based evidence sequences.

That asymptotic result leaves several practical questions open. Our implementation fixes the tree depth and tabulates mixtures to a finite horizon. Its power depends on the chosen splits, and an error that is invisible within those cells can remain undetected. Fine partitions also become expensive with many covariates. The current experiments do not establish adaptation to sparse high-dimensional structure or optimal finite-sample alarm times. Measurement intervals and approximate ties require further modeling: treating two nearby values as equal changes the hypothesis being tested.

Regression inversion offers another direction for development. The coverage statement assumes conditional quantile correctness, and computation must retain the exact-fit faces where tie factors change. A certified inversion procedure remains to be developed. Other useful extensions include growing trees with finite-horizon power guarantees, drifting alternatives, and broader comparisons with feature-learning methods under a common null. The present experiments suggest that those comparisons should report detection times as well as rejection proportions: adapting to an unknown partition can help, but it has a measurable cost.

\appendix

\section{Additional mathematical checks}

\paragraph{Allowance for a known tolerance.}

Suppose instead that $\Pp(Z_t=d\mid\G_t)\leq p_d+\delta_t$ for a known, predictable $\delta_t\geq0$. Then
\[
\frac{1+\lambda_t(\ind\{Z_t=d\}/p_d-1)}
{1+\lambda_t\delta_t/p_d}
\]
has conditional expectation at most one and is therefore valid. Here $\delta_t$ must be a known bound available before the response. If it is estimated, uncertainty in that estimate needs to be included in the construction.

\paragraph{Examples of valid factors.}

For $\tau=1/2$, the vector $(1.2,0.2,0.8)$ has null support function one and is valid. The factor $(1.2,0.8,0.8)=e^L_{0.2}$ dominates it. The degenerate null concentrated at the forecast also forces $e_T\leq1$, so a large payoff on ties cannot produce unrestricted evidence.

\paragraph{Numerical verification.}

We compared the closed-form information value with constrained numerical maximization of $\sum_zp_z\log e_z$ for 250 independently generated $(\tau,p)$ combinations. All runs satisfied the verification tolerance $10^{-6}$ under the four vertex constraints. A separate implementation using an analytic objective gradient on 1000 additional combinations gave maximum information-value discrepancy $5.1\times10^{-13}$ and maximum constraint violation $8.9\times10^{-16}$. The reproducibility files include the trial results, seeds, and data checksum.

An incomplete-beta identity provided a separate check of the mixture recurrence, with maximum absolute log discrepancy $8.6\times10^{-14}$. Enumeration of all 26 depth-3 prunings on 450 streams at $\tau\in\{0.2,0.5,0.8\}$ agreed with the recursion to $3.6\times10^{-15}$ in log capital, and each stream satisfied the oracle bound. We also checked the resolution-loss inequality in 500 randomly generated positive-probability configurations. These computations check the formulas and their implementation alongside the analytical arguments.

\section{Reproducibility and literature audit}

The numerical supplement contains the analysis scripts, pinned dependency versions, and saved results needed to reproduce the tables. It includes binomial uncertainty intervals, fitted coefficients, subgroup sign frequencies, evidence-path plots, and the public AER/Rdatasets data with provenance and a SHA-256 checksum. The auxiliary script downloads the data only if the included copy is absent.

The primary tree and auxiliary finite-grid experiments have separate seeds and output files. The matched-null comparison reconstructs the primary streams and records individual alarm times, paired rejection differences, and restricted mean alarm times. Separate files record the held-out tree analysis and the forecast perturbation.

The literature review covered discrete and mixed-outcome quantile inference, quantile regression under shift, sequential quantile inference, and e-backtesting and constraint-based e-values. In comparing the closest sources, we used the atom qualification in \citet{antonov}, the quantile example in \citet{larsson}, and the one-sided construction in \citet{wang}. These comparisons motivate the attribution of the general complete-class and sequential-validity arguments to existing theory. A separate literature-audit note records the scope of the review and the proposed contribution. Some recent sources are preprints, so the review should be understood as a scoped comparison rather than an exhaustive priority claim.

\section*{Declarations}

\paragraph{Funding.}
The author reports no funding for this work.

\paragraph{Competing interests.}
The author declares no conflicts of interest.

\paragraph{Data and code availability.}
The DoctorVisits data are publicly available at the URL in \citet{aer}. The accompanying reproducibility archive contains the analysis code, numerical results, dependency versions, and data checksum.

\paragraph{Acknowledgment of AI assistance.}
The author used AI tools to assist with LaTeX preparation and numerical verification. The research idea originated with the author, who developed the work and takes responsibility for the manuscript’s content and results.


\begin{thebibliography}{99}

\bibitem[AER documentation(n.d.)]{aer}
AER package documentation (n.d.).
DoctorVisits: Australian Health Service Utilization Data.
\href{https://vincentarelbundock.github.io/Rdatasets/doc/AER/DoctorVisits.html}
{Documentation and provenance}, accessed October 11, 2026.

\bibitem[Antonov et al.(2026)]{antonov}
Antonov, I., Mukherjee, S., Pibernik, R., and Choe, Y. J. (2026).
Bet on Features: Anytime-Valid and Feature-Aware Auditing of Conditional Quantile Forecasters.
Preprint,
\href{https://arxiv.org/abs/2607.11653}{arXiv:2607.11653}.

\bibitem[Casgrain et al.(2024)]{casgrain}
Casgrain, P., Larsson, M., and Ziegel, J. (2024).
Sequential testing for elicitable functionals via supermartingales.
\emph{Bernoulli}, 30(2), 1347--1374.
\href{https://doi.org/10.3150/23-BEJ1634}
{doi:10.3150/23-BEJ1634}.

\bibitem[Chernozhukov et al.(2020)]{chernozhukov}
Chernozhukov, V., Fern\'andez-Val, I., Melly, B., and W\"uthrich, K. (2020).
Generic inference on quantile and quantile effect functions for discrete outcomes.
\emph{Journal of the American Statistical Association}, 115(529), 123--137.
\href{https://arxiv.org/abs/1608.05142}
{Preprint: arXiv:1608.05142}.

\bibitem[Clerico(2026)]{clerico}
Clerico, E. (2026).
Sequential testing of conditionally constrained hypotheses.
Preprint,
\href{https://arxiv.org/abs/2606.06769}{arXiv:2606.06769}.

\bibitem[Cover(1991)]{cover}
Cover, T. M. (1991).
Universal portfolios.
\emph{Mathematical Finance}, 1(1), 1--29.
\href{https://doi.org/10.1111/j.1467-9965.1991.tb00002.x}
{doi:10.1111/j.1467-9965.1991.tb00002.x}.

\bibitem[Cover and Ordentlich(1996)]{coverord}
Cover, T. M. and Ordentlich, E. (1996).
Universal portfolios with side information.
\emph{IEEE Transactions on Information Theory}, 42(2), 348--363.
\href{https://isl.stanford.edu/~cover/papers/cover_ordentlich_96.pdf}
{Author-hosted article}.

\bibitem[Howard and Ramdas(2022)]{howard}
Howard, S. R. and Ramdas, A. (2022).
Sequential estimation of quantiles with applications to A/B testing and best-arm identification.
\emph{Bernoulli}, 28(3), 1704--1728.
\href{https://doi.org/10.3150/21-BEJ1388}
{doi:10.3150/21-BEJ1388}.

\bibitem[Koenker and Bassett(1978)]{koenker}
Koenker, R. and Bassett, G., Jr. (1978).
Regression quantiles.
\emph{Econometrica}, 46(1), 33--50.
\href{https://doi.org/10.2307/1913643}
{doi:10.2307/1913643}.

\bibitem[Larsson et al.(2026)]{larsson}
Larsson, M., Ramdas, A., and Ruf, J. (2026).
Testing hypotheses generated by constraints.
\emph{Mathematics of Operations Research}, published online.
\href{https://doi.org/10.1287/moor.2025.1129}
{doi:10.1287/moor.2025.1129}.

\bibitem[Orabona and Jun(2024)]{orabona}
Orabona, F. and Jun, K.-S. (2024).
Tight concentrations and confidence sequences from the regret of universal portfolio.
\emph{IEEE Transactions on Information Theory}, 70(1), 436--455.
\href{https://doi.org/10.1109/TIT.2023.3330187}
{doi:10.1109/TIT.2023.3330187};
\href{https://arxiv.org/abs/2110.14099}{arXiv:2110.14099}.

\bibitem[Ramdas et al.(2022)]{ramdasexchange}
Ramdas, A., Ruf, J., Larsson, M., and Koolen, W. M. (2022).
Testing exchangeability: Fork-convexity, supermartingales and e-processes.
\emph{International Journal of Approximate Reasoning}, 141, 83--109.
\href{https://doi.org/10.1016/j.ijar.2021.06.017}
{doi:10.1016/j.ijar.2021.06.017}.

\bibitem[Wang et al.(2026)]{wang}
Wang, Q., Wang, R., and Ziegel, J. (2026).
E-backtesting.
\emph{Management Science}, 72(6), 4952--4973.
\href{https://doi.org/10.1287/mnsc.2023.01659}
{doi:10.1287/mnsc.2023.01659};
\href{https://arxiv.org/abs/2209.00991}{author manuscript}.

\bibitem[Waudby-Smith and Ramdas(2024)]{waudby}
Waudby-Smith, I. and Ramdas, A. (2024).
Estimating means of bounded random variables by betting.
\emph{Journal of the Royal Statistical Society Series B}, 86(1), 1--27.
\href{https://doi.org/10.1093/jrsssb/qkad009}
{doi:10.1093/jrsssb/qkad009}.

\bibitem[Willems et al.(1995)]{willems}
Willems, F. M. J., Shtarkov, Y. M., and Tjalkens, T. J. (1995).
The context-tree weighting method: Basic properties.
\emph{IEEE Transactions on Information Theory}, 41(3), 653--664.
\href{https://doi.org/10.1109/18.382012}
{doi:10.1109/18.382012}.

\end{thebibliography}
\end{document}